\documentclass[11pt]{article}
\usepackage[english]{babel}
\usepackage[utf8]{inputenc}
\usepackage{amsmath,amsfonts,amssymb,amsthm,mathrsfs,dsfont}

\usepackage[mono=false]{libertine}
\useosf

\usepackage{wasysym}

\usepackage[dvipsnames]{xcolor}
\usepackage[a4paper,vmargin={3.1cm,3.1cm},hmargin={2.4cm,2.4cm}]{geometry}
\usepackage[font=sf, labelfont={sf,bf}, margin=1cm]{caption}

\usepackage{lscape}
\usepackage{pdflscape}
\usepackage{graphicx}

\usepackage[font=footnotesize,labelfont=bf]{caption}
\usepackage{subfig} 
\usepackage{epsfig}
\usepackage{latexsym}
\usepackage{stmaryrd}
\usepackage{ae,aecompl}

\usepackage[english]{babel}
 \usepackage[colorlinks=true]{hyperref}
\usepackage{pstricks}
\usepackage{enumerate}
\usepackage[normalem]{ulem}

\usepackage{hyperref}
\usepackage[hyphenbreaks]{breakurl}

\newtheorem{theorem}{Theorem}[]

\newtheorem{prop}[]{Proposition}

\newtheorem{lemma}[]{Lemma}

\newtheorem{conjecture}[]{Conjecture}
\newtheorem{defi}[theorem]{Definition}
\theoremstyle{definition}

\renewcommand{\leq}{\leqslant}
\renewcommand{\geq}{\geqslant}

\renewcommand{\P}{\mathbb{P}}

\newcommand{\defeq}{=} 

\newcommand{\Cor}{\ensuremath{\hbox{\textup{Cor}}}}

\renewcommand{\c}[2]{\ensuremath{[#1 | #2]}}

\newcommand{\I}{\mbox{}}

\newcommand{\motif}[2]{\ensuremath{\textbf{Patt}_{#2}(#1)}}

\newcommand*\samethanks[1][\value{footnote}]{\footnotemark[#1]}

\title{{\bf On cases where Litt's game is fair}}

\author{
A-L. Basdevant \thanks{Sorbonne Université.\hfill  \href{mailto:}{\texttt{anne.laure.basdevant@normalesup.org}}} \ \& \
Olivier Hénard \thanks{Universit\'e Paris-Saclay.
\hfill  \href{mailto:olivier.henard@universite-paris-saclay.fr}{\texttt{olivier.henard@universite-paris-saclay.fr}} \\
 \hspace*{\fill}  \href{mailto:edouard.maurel-segala@universite-paris-saclay.fr}{\texttt{edouard.maurel-segala@universite-paris-saclay.fr}} \\ 
\hspace*{\fill}  \href{mailto:arvind.singh@universite-paris-saclay.fr}{\texttt{arvind.singh@universite-paris-saclay.fr}} 
}
\ \& \
Édouard Maurel-Ségala \samethanks
\ \& \
Arvind Singh \samethanks
}

\begin{document}

\maketitle

\begin{abstract}
A fair coin is flipped $n$ times, and two finite sequences of heads and tails with the same length are given, say $A$ and $B$. Each time the word $A$ appears in the sequence of fair coin flips, Alice gets a point, and each time the word $B$ appears, Bob gets a point. Who is more likely to win?  This puzzle is a slight extension of Litt's game \cite{LittX} that recently set Twitter abuzz. 
We show that Litt's game is fair for any value of $n$ and any two words that have the same auto-correlation structure by building up a bijection that exchanges Bob and Alice scores; the fact that the inter-correlations do not come into play in this case may come up as a surprise. 
\end{abstract}

\selectlanguage{english}

\maketitle

\vspace{1.3cm}

\section{Introduction and main results}

In \cite{LittX}, Litt came up with the following puzzle: a fair coin is tossed $n$ times, Alice gets a point each time the sequence HH appears, while Bob scores a point each time the sequence HT appears (and these sequences may be overlapping). The game may result in a win for Alice, a win for Bob or a tie (the possibilities being exclusive). Who is more likely to win, Alice or Bob ? 
The perhaps surprising answer (if one judges by the Twitter poll) is that:
\begin{enumerate}
    \item for any size $n\geq 3$, Bob is more likely to win than Alice,
    \item for $n$ large, it holds:
$$2 \times \big(\P_n( \mbox{Bob wins})-\P_n(\mbox{Alice wins})\big) \sim \P_n(\mbox{Tie}) \sim \frac{1}{\sqrt{\pi n}}.$$
\end{enumerate} 
Several elements of proofs for the first fact quickly emerged on Twitter \cite{RameshX1, RameshX2, SelbyX, MalloX}, then more formal proofs of both facts appeared  on the arXiv \cite{Zeilberger24, Segert24}. Nica \cite{NicaYouTube}  recorded a YouTube video to introduce the problem to a wider audience. Zeilberger online journal \cite{Zeilbergerjournal} maintains a list of contributions and questions on the problem; among the proofs given, we would like to advertise an early probabilistic proof : in \cite{RameshX1}, Ramesh builds a ``fair majorant" for the score difference of Alice and Bob that consists in a delayed simple random walk; the score difference of Alice and Bob  is a deterministic function of this walk, at distance at most 1 below it, see right after the bibliography for more details. A rigorous proof of both facts 1 and 2 is then at an easy reach. Another by-product of his approach is that giving Alice an initial advantage of only one point reverses the statement of fact 1 in a strong sense: for any size $n \geq 0$, Alice then has probability strictly greater than $1/2$ of winning. In particular, Alice is more likely to win than Bob. This approach has been detailed in Grimmett \cite{Grimmett24}, who recently rediscovered the arguments of Ramesh.

\medskip

Considering a contest between HT and TH is also possible, but arguably less interesting: flipping the sequence of tosses, or reading the sequence in the reverse order  exchanges Alice and Bob points, which results in a fair game. Our aim in this note is to uncover more hidden symmetries within the words ensuring the game is again tied up.

\medskip


 For $\ell\geq 0$, a \emph{word} of length $\ell$ is a sequence $A=(a_1,\ldots,a_\ell)$ in $\{\text{H},\text{T}\}^\ell$. For $X_n:=(\varepsilon_k)_{1\leq k\leq n}$ a finite sequence in $\{\text{H},\text{T}\}^n$, we denote by
$$N_A(X_n):=|\{\ell\leq k\leq n , (\varepsilon_{k-\ell+1},\ldots,\varepsilon_k)=(a_1,\ldots,a_\ell)\} |$$
the number of occurrences of the word $A$ in the sequence $X_n$.
Let $A$ and $B$ be two words of length $\ell$.
In the generalized Litt's game, Alice wins if $N_A(X_n)>N_B(X_n)$, Bob wins if $N_B(X_n)>N_A(X_n)$, and this is a Tie otherwise.
A key quantity encoding the intersections of $A$ and $B$ is
the correlation of $A$ and $B$, which may be represented as a subset of integers or simply as a number (the base 2 expansion of the subset):
$$\c{A}{B} \defeq \sum_{k\in \Cor(A,B)} 2^k \quad\text{ where }\quad \mbox{Cor}(A,B) \defeq\{1\leq k \leq \ell-1,\; (a_{\ell-k+1},\ldots,a_\ell)=(b_1,\ldots,b_k)\}.$$
To be more specific, we shall call \emph{inter-correlation} the correlation of two distinct words (beware the order matters), and \emph{auto-correlation} the correlation of a word with itself.
Our main result in this note is that Litt's game, which is played with a uniform $X_n$, is fair for words $A$ and $B$ with the same auto-correlation, regardless of their inter-correlation. 

\begin{theorem}\label{bijection} Let $A$, $B$ two words of length $\ell$ such that $\c{A}{A}= \c{B}{B}$.
Assume that, under $\P_n$, the letters $(\varepsilon_i)_{1\leq i\leq n}$ of $X_n$ form an i.i.d. sequence 
with the uniform distribution on $\{\text{H},\text{T}\}$.
Then for each $n\geq 1$, $(N_A(X_n),N_B(X_n))$  and $(N_B(X_n),N_A(X_n))$ have the same distribution. In particular, for any $n\geq 1$,
$$\P_n( \emph{Bob wins})=\P_n(\emph{Alice wins}).$$
\end{theorem}

For instance, the words $A =$ HHTHTH and $B=$ HTTTHH have the same auto-correlation 2 hence Litt's game played with $A$ and $B$ is fair despite the apparent lack of symmetry between these words. 

\medskip

It is then possible to estimate the quantity in the last equality, or equivalently the probability of a tie $\P_{n}(\text{Tie})$, using a Local Central Limit theorem for sums of weakly dependent random variables. We do not venture into this for the following reason : In a recent breakthrough announced online \cite{Nica}, Nica (together with Janson) pointed to a generic method based on Edgeworth expansions to tackle the case of words $A$ and $B$ with possibly distinct auto-correlations; the sketch of proof, together with some basic moment computations, hints to the following asymptotic estimates: for $A \neq B$, as $n$ gets large, the following asymptotics:
\begin{align*}
\P_n( \mbox{Bob wins}) -\P_n(\mbox{Alice wins}) & =   \frac{\c{A}{A} -\c{B}{B}}{ \sqrt{2^\ell + \c{A}{A} +\c{B}{B} - \c{A}{B} -\c{B}{A} } } \cdot \frac{1}{2 \sqrt{\pi n}}+o\Big(\frac{1}{\sqrt{n}}\Big),\\ \P_n( \mbox{Tie}) & = \frac{2^\ell}{\sqrt{2^\ell + \c{A}{A} +\c{B}{B} - \c{A}{B} - \c{B}{A} }} \cdot \frac{1}{2 \sqrt{\pi n 
}}+o(\frac{1}{\sqrt{n}}\Big)
\end{align*}
hold as soon as the denominators on both RHS are non null (the four\footnote{(TH$^{\ell-1}$,H$^{\ell-1}$T) and its siblings, for which $N_{A}(X_n)-N_B(X_n)$ belongs to $\{-1,0,1\}$.} pair of words giving a null denominator are associated with degenerated cases). Shortly after posting a first version of this work, Segert communicated to us a proof, available online \cite{Segert24+} as an addendum to his arXiv paper, of a general statement that entails both estimates.
In case $\c{A}{A} =\c{B}{B}$, Theorem \ref{bijection} refines on the first formula by stating that the quantity is, in fact, identically null for all~$n$. 

\medskip

A key feature of the first formula  is that $\P_n( \mbox{Bob wins}) -\P_n(\mbox{Alice wins})$ has the sign of $\c{A}{A} -\c{B}{B}$. In light of this observation together with some extensive numerical computations\footnote{the conjecture has been checked for all words $A$ and $B$ of size $\leq 10$ and for $n \leq 35$.}, we are lead to state the following remarkable conjecture for a \emph{fixed length} $n$ of the underlying word $X_n$. We use $\Delta$ to denote the symmetric difference of two sets.

\begin{conjecture}
\label{conj-distinct-correl} 
Assume $\c{A}{A}>\c{B}{B}$. Then, for each $n \geq n_0:=2 \ell- \max\{ \Cor(A) \Delta \Cor(B) \}$, 
$$\P_n(\emph{Alice wins}) < \P_n(\emph{Bob wins}) < \frac{1}{2} < \P_n(\emph{Alice wins}) + \P_n(\emph{Tie}).$$  
\end{conjecture}

Some comments are in order. First, the restriction to $n \geq n_0$ is necessary for the first inequality only (the most important one indeed), and we have equality for $n<n_0$. Also, the last two inequalities are equivalent (and valid for each $n \geq 0$), but we like to phrase the conjecture this way, because of the following vivid interpretation of this set of inequalities: even if Alice looses in the standard version of the game (first inequality), if we were to give Alice \emph{one additional point} \footnote{beware this is distinct from starting the random word $X_n$ with an $A$}, then her new winning probability would exceed $1/2$ for each $n$ (last inequality); in particular, Alice would be advantaged over Bob at each subsequent time. Conjecture \ref{conj-distinct-correl} also suggests that the advantage of Bob versus Alice may be put in bijection with a subset of the Tie event, a key feature of the proof by Ramesh indeed. Conjecture \ref{conj-distinct-correl}  is valid for the original Litt's game (HH versus HT) by Ramesh's proof. Also, it is in line with the asymptotics above: to see this, notice the last two inequalities in the conjecture are equivalent to 
$$|\P_n(\text{Alice wins}) - \P_n(\text{Bob wins})| < \P_n(\text{Tie}),$$
which may be checked from the conjectured asymptotics because of the inequality $|\c{A}{A} -\c{B}{B}| \leq 2^{\ell}-2$ that is true for every pair of words $A$,$B$.

\medskip

We should also point that the topic of pattern matching and overlaps has been the subject of many investigations, starting with the non-transitive \emph{Penney Ante game} named after Penney \cite{Penney-69} (and famously solved by Conway) in which one is asked to compute the probability that $A$ appears before $B$ in a sequence of fair coin flips, see \cite{Guibas-Odlyzko-81, Li-80} and the references therein.

Our fixed length conjecture is reminiscent of an important fixed length result in the area of pattern matching : the fact that the number of words $X$ of length $n$ avoiding $A$ is larger than the number of words avoiding $B$ if $[A,A] > [B,B]$ for each $n \geq n_0$ defined in term of $A$ and $B$ as in Conjecture 1, and with equality before. For $n$ large enough (larger than an indefinite constant), the statement has been stated and proved by Guibas and Odlyzko \cite{Guibas-Odlyzko-81} and independently by Blom \cite{Blom84}, while the result for $n \geq n_0$ has been established by Månnson \cite{Mansson02} after a series of works on the topic \cite{Eriksson97,Mansson99}.

\section{Proof of Theorem \ref{bijection}}

The auto-correlation and inter-correlation of two words are quantities that appear naturally when we look at the probability of one word appearing before another in a sequence of coin flips. The formal definition, which we repeat below, is the following. 
  
\begin{defi}Let $A$, $B$ two words of length $\ell$. We define the indices of inter-correlation of $A$ and $B$ by
  $$\Cor(A,B)=\{1\le k \le \ell-1, (a_{\ell-k+1},\ldots,a_\ell)=(b_1,\ldots,b_k)\}.$$
  We write $\Cor(A)$ to denote $\Cor(A,A)$. We define the inter-correlation $\c{A}{B}$  of $(A,B)$ by
  $$\c{A}{B} =\sum_{k\in \Cor(A,B)} 2^k$$
 and the auto-correlation of $A$ as  $\c{A}{A}$.
\end{defi}

\noindent Let us make some remarks about this definition
  \begin{itemize}
      \item The number $\c{A}{B}$ is not in general equal to $\c{B}{A}$.
      \item For all $A,B$ of length $\ell$, $\c{A}{B} \in \llbracket 0; 2^{\ell}-2 \rrbracket$.
      \item For every words $A,B,C,D$ of length $\ell$, $\c{A}{B}=\c{C}{D}$ if and only if $\Cor(A,B)=\Cor(C,D)$.
  \end{itemize}

Fix any two words $A$ and $B$ with length $\ell$ and the same auto-correlation. To prove that $(N_A(X_n),N_B(X_n))$  has the same law as $(N_B(X_n),N_A(X_n))$, we prove the existence of a bijection $\phi$ from $\{$H,T$\}^n$ to $\{$H,T$\}^n$ such that, for any sequence $X_n\in\{$H,T$\}^n$, $(N_A(X_n),N_B(X_n))=(N_B(\phi(X_n)),N_A(\phi(X_n)))$.
The rest of the paper is devoted to the construction of $\phi$.

\medskip

We introduce some additional notation. If $C,D$ are two words, $CD$ will be the concatenation of $C$ and $D$. Besides, for two words $C,D$ of length $\ell$ and $m\in \Cor(C,D)$, we denote $C^{m}D$ the word of length $2\ell-m$ beginning by $C$ and ending by $D$. For example, if $C = \underline{\text{HTTHTH}}$, $D = \overline{\text{THTHHH}}$, we have $C^2D= \underline{\text{HTTH}} \underline{\overline{\mathbf{\text{TH}}}} \overline{\text{THHH}}$. We extend this notation to $k$ words $C_1,\ldots,C_k$ of length $\ell$ and $m_1,\ldots,m_k\geq1$ in the obvious way: if $m_i\in \Cor(C_i,C_{i+1})$ for each $i$, we define the word of length $k\ell-\sum m_i$
\begin{equation}\label{notE}
Y = C_{1}\I^{m_1}C_{2}\I^{m_2}C_{3}\I^{m_3} \ldots C_{k-1}\I^{m_{k-1}}C_{k}.
\end{equation}
 
\begin{defi} Let $A,B$ be two words of length $\ell$. We call \emph{overlap} of $A$ and $B$ a word $Y$ of the form \eqref{notE} where  $C_1,\ldots,C_k \in \{A,B\}^k$. We denote by $\mathcal{E}(A,B)$ the set of all \emph{overlaps} of $A$ and $B$.
\end{defi}
Note that do not accept $m_i=0$ in \eqref{notE}. In particular, the concatenation $Y=AB$ of $A$ and $B$ may not be in $\mathcal{E}(A,B)$.
We notice also that, for $Y\in \mathcal{E}(A,B)$, the expression of $Y$ in the form given by \eqref{notE} may be not unique. We call the maximal decomposition of $Y$ the one such that, if 
$Y=C_{1}\I^{m_1}C_{2}\I^{m_2}C_{3}\I^{m_3} \ldots C_{k-1}\I^{m_{k-1}}C_{k}$, we have
$$N_A(Y)=|\{ 1\le i \le k, C_i=A\}| \qquad N_B(Y)=|\{ 1\le i \le k, C_i=B\}|.$$
\noindent We have the following decomposition of a word.
\begin{defi} Let $Y$ a word. There exists a unique way to write $Y$ in the form
\begin{equation}\label{decomp}
Y = X_0 E_1 X_1 E_2 \ldots X_{k-1} E_k X_{k}
\end{equation}
such that
\begin{itemize}
    \item the words $E_1,\ldots, E_k$ belong to $\mathcal{E}(A,B)$.
    \item For all $i\in \llbracket 0,k\rrbracket$, neither $A$ nor $B$ appears in $X_i$ (word $X_i$ may be empty).
    \item $N_A(Y)=\sum_{i=1}^k N_A(E_i)$ and $N_B(Y)=\sum_{i=1}^k N_B(E_i)$.
\end{itemize}
We call \emph{pattern} of the word $Y$ with respect to $A$ and $B$ the words $(E_1,\ldots,E_k)$ which appears in \eqref{decomp} and write 
$\motif{Y}{A,B} = (E_1,E_2,\ldots, E_k)$.
\end{defi}

\begin{prop} Let $A,B$ be two words with the same auto-correlation, and let $\phi:\mathcal{E}(A,B) \to \mathcal{E}(A,B)$ be defined in the following way. If $Y:=C_{1}\I^{m_1}C_{2}\I^{m_2} \ldots C_{k-1}\I^{m_{k-1}}C_{k}$ with $C_i \in \{A,B\}$, we set 
$$\phi(Y)\defeq\bar{C}_{k}\I^{m_{k-1}}\bar{C}_{k-1}\I^{m_{k-2}} \ldots \bar{C}_{2}\I^{m_{1}}\bar{C}_{1}$$
where $\bar{C}_i=A$ if $C_i=B$ and  $\bar{C}_i=B$ if $C_i=A$.
Then $\phi$ is well-defined, it is independent of the decomposition chosen for $Y$, it is an involution and $\phi(Y)$ has the same length as $Y$. Moreover, we have $(N_A(Y),N_B(Y))=(N_B(\phi(Y)),N_A(\phi(Y)))$.
\end{prop}

\begin{proof}We first prove  that $\phi(Y)$ is well defined \emph{i.e.} if $m_i\in \text{Cor}(C_i,C_{i+1})$, then $m_i\in \text{Cor}(\bar{C}_{i+1},\bar{C}_i)$. Recall that we assume that $A,B$ have the same auto-correlation, \emph{i.e.} $\text{Cor}(A)=\text{Cor}(B)$. We have two cases: 
\begin{itemize}
    \item Either $C_i=C_{i+1}$, for example $C_i=A$. Then $\bar{C}_i=\bar{C}_{i+1}=B$. Hence, we have $\text{Cor}(C_i,C_{i+1})=\text{Cor}(A)=\text{Cor}(B)=\text{Cor}(\bar{C}_{i+1},\bar{C}_i)$. 
    \item Or $C_i\neq C_{i+1}$, for example $(C_i,C_{i+1})=(A,B)$. Then $(\bar{C}_{i+1},\bar{C}_i)$ is also equal to $(A,B)$ and so $\text{Cor}(C_i,C_{i+1})=\text{Cor}(\bar{C}_{i+1},\bar{C}_i)$.
\end{itemize}
Note that $\phi$ does not depend of the decomposition chosen for $Y$. To justify this claim, it is enough to consider the case of words $Y$ with two decompositions $Y=C_{1}\I^{m_1}C_{2}\I^{m_2}C_{3}$ and $Y=C_{1}\I^{m}C_{3}$. Note that $|Y|=2\ell-m=3\ell-m_1-m_2$. Applying $\phi$ with the first decomposition we get $\phi(Y)=\bar{C}_{3}\I^{m_2}\bar{C}_{2}\I^{m_1}\bar{C}_{1}$ and since $|\phi(Y)|=|Y|<2\ell$,  necessarily, $\bar{C}_{3}$ and  $\bar{C}_{1}$ overlap in the writing of $\phi(Y)$ and thus we also have
$\phi(Y)=\bar{C}_{3}\I^{m}\bar{C}_{1}$. 
The fact that $\phi$ is an involution is clear.

We write now $Y$ with its maximal decomposition. By construction, we directly get that
$N_A(\phi(Y)) \geq N_B(Y)$ and
$N_B(\phi(Y)) \geq N_A(Y)$. 
We claim that, by maximality of the decomposition, there is equality in these two inequalities. Indeed, consider the case of words $Y$ whose maximal decomposition consist in two words: $C_{1}\I^{m}C_{2}$. Among those words $Y$, only the words where $C_1=C_2$ have to be considered.  If $Y_0= C_1 \I^{m}C_1$, it holds $\phi(Y_0) = \bar C_1 \I^{m} \bar C_1$. Now, assume that
$\phi(Y_0) = \bar C_1 \I^{m_1} C_3 \I^{m_2}  \bar C_1$ for some $C_3 \in \{A,B\}$. The map $\phi$ being independent of the decomposition, we get $Y_0  = \phi(\phi(Y_0)) = C_1 \I^{m_1} \bar{C_3} \I^{m_2} C_1$, which has a distinct count of the $\bar{C_3}$-word; this negates the fact that $Y_0= C_1 \I^{m}C_1$ is the maximal decomposition in the first place.
\end{proof}

The function $\phi$ defined in the previous proposition can be extended to a function on patterns. For $(E_1,\ldots,E_k)\in\mathcal{E}(A,B)^k$, we set
$$\phi(E_1,\ldots,E_k) \defeq (\phi(E_k),\ldots,\phi(E_1))$$
which defines an involution on $\mathcal{E}(A,B)^k$.
Given a pattern $M=(E_1,\ldots,E_k)\in \mathcal{E}(A,B)^k$ and $n\geq 1$, we define $L_M(n)$ as the number of words of length $n$ with pattern $M$: 
\begin{equation*}
 L_M(n)\defeq|\{ \hbox{word $Y$} ,\; |Y| = n \hbox{ and } \motif{Y}{A,B} = M \} |.\end{equation*}
To prove Theorem \ref{bijection}, it is sufficient to show that, for any pattern $M$ and any $n\geq 1$, we have
\begin{equation}\label{egalitemot}
 L_M(n)=L_{\phi(M)}(n).  \end{equation}
Indeed, if $Y$ is a word such that $\motif{Y}{A,B} = M$ and $Z$ is a word such that $\motif{Z}{A,B} = \phi(M)$, we have 
$$(N_A(Y),N_B(Y))=(\sum_{i=1}^k \! N_A(E_i),\sum_{i=1}^k \!N_B(E_i))=(\sum_{i=1}^k \! N_B(\phi(E_i)),\sum_{i=1}^k \!N_A(\phi(E_i)))=(N_B(Z),N_A(Z)).$$
In view of Equality \eqref{egalitemot}, we conclude that
\begin{eqnarray*}
    P\big((N_A(X_n),N_B(X_n))=(a,b)\big) \; =  \; \frac{1}{2^n}\hspace{-0.5cm}\sum_{\substack{\hbox{\scriptsize{pattern $M$ s.t.}}\\ \hbox{\tiny{$(N_A(M),N_B(M))=(a,b)$}}}}\hspace{-0.5cm}L_M(n)  \; &=&  \; \frac{1}{2^n}\hspace{-0.5cm}\sum_{\substack{\hbox{\scriptsize{pattern $M$ s.t.}}\\\hbox{\tiny{$(N_A(M),N_B(M))=(a,b)$}}}}\hspace{-0.5cm}L_{\phi(M)}(n)\\
    &=& P\big((N_B(X_n),N_A(X_n))=(a,b)\big).
\end{eqnarray*}
In order to establish \eqref{egalitemot}, we prove the more slightly more precise result: 
\begin{lemma}
\label{1}
 For any pattern $M=(E_1,\ldots,E_k)$, for any   $I=(i_0,\ldots,i_k)$, set 
   $$
L^{M}(I) \defeq |\{ Y=X_0E_1X_1\ldots E_kX_k :  \motif{Y}{A,B} = M; \forall j, \, |X_j| = i_j \} |.$$
Then $$|L^{M}(I)|=|L^{\phi(M)}(I')|$$
where $I' = (i_k,\ldots,i_0)$.
\end{lemma}

\begin{proof}
 
 Let us remark that we have
$$L^M(I)=L^{E_1}(i_0,0)\left(\prod_{j=1}^{k-1} L^{(E_j,E_{j+1})}(0,i_j,0)\right)L^{E_k}(0,i_k)$$
and 

$$L^{\phi(M)}(I')=L^{\phi(E_k)}(i_k,0)\left(\prod_{j=1}^{k-1} L^{(\phi(E_{j+1}),\phi(E_{j}))}(0,i_j,0)\right)L^{\phi(E_1)}(0,i_0).$$
Moreover, the value of $L^{(E_j,E_{j+1})}(0,i_j,0)$ only depends on $i_j$, on the word ending $E_j$ and on the word beginning $E_{j+1}$. For example, if $E_j$ ends with an $A$ and $E_{j+1}$ starts with a $B$, we have
$$L^{(E_j,E_{j+1})}(0,i_j,0)=L^{(A,B)}(0,i_j,0).$$
 Now, if  $E_j$ ends with an $A$ and $E_{j+1}$ starts with a $B$, then $\phi(E_{j+1})$ ends with an $A$ and $\phi(E_{j})$ starts with an $B$. Thus, in this case, we directly get 
 $$L^{(\phi(E_{j+1}),\phi(E_{j}))}(0,i_j,0)=L^{(E_j,E_{j+1})}(0,i_j,0)=L^{(A,B)}(0,i_j,0).$$
The situation is more involved when  $E_j$ ends with the same word than $E_{j+1}$ starts with, let say the word $A$. Then indeed
 $$L^{(\phi(E_{j+1}),\phi(E_{j}))}(0,i_j,0)=L^{(B,B)}(0,i_j,0) \mbox{ while } L^{(E_j,E_{j+1})}(0,i_j,0)=L^{(A,A)}(0,i_j,0).$$
Besides, if $E_1$ starts  with $A$, then $\phi(E_1)$ ends with a $B$, and we have 
$$L^{\phi(E_{1})}(0,i_0)=L^{(B)}(0,i_0) \mbox{ while }  L^{E_1}(i_0,0)=L^{(A)}(i_0,0),$$
and a similar assertion holds for $L^{E_k}(0,i_k)$ and  $L^{\phi(E_k)}(i_k,0)$.
Combining all these remarks, we see that Lemma \ref{1} will be proved as soon as we establish the following  equality: for each $i \geq 0$,
\begin{equation}\label{e:recurrence}
L^{(A,A)}(0,i,0)=L^{(B,B)}(0,i,0) \qquad\mbox{and}\qquad  L^{(B)}(0,i)=L^{(A)}(i,0).
\end{equation}
We prove \eqref{e:recurrence} by induction on $i$. The proposition clearly holds for $i=0$. Assume it holds for $k\le i-1$. Thus,  for any pattern $M$, if $I=(i_0,\ldots,i_k)$ with $i_j<i$ for all $j$, we get $$|L^{M}(I)|=|L^{\phi(M)}(I')|.$$ 
Let us now prove that $L^{(B)}(0,i)=L^{(A)}(i,0)$.
Writing $|I|= \sum_{j=0}^k i_j$ if $I=(i_0, \ldots, i_k)$ and
$|M|= \sum_{i=1}^k |E_i|$ for the length of the pattern
$M=(E_1,\ldots, E_k)$, we find that
\begin{eqnarray*}
L^{(A)}(i,0)=
   &= &|\{ \hbox{words $X A$} :\; |X| = i \hbox{ and } \motif{X A}{A,B} = A \} |\\
    &=&2^i-\sum_{k\geqslant 1}\hspace{0.1cm}\sum_{\substack{M=(E_1,\ldots,E_k)\neq (A)\\\hbox{\scriptsize{$E_k\mbox{ ends with }A$}}}} \hspace{0.3cm} \sum_{\substack{I=(i_0,\ldots,i_{k-1},0)\\|I|+|M|=i+|A|}}|L^M(I)|\\
     &=&2^i-\sum_{k\geqslant 1}\hspace{0.1cm}\sum_{\substack{M=(E_1,\ldots,E_k)\neq (A)\\\hbox{\scriptsize{$E_k\mbox{ ends with }A$}}}} \hspace{0.3cm} \sum_{\substack{I=(i_0,\ldots,i_{k-1},0)\\|I|+|M|=i+|A|}}|L^{\phi(M)}(I')|.
\end{eqnarray*}
At this point, we use the recurrence assumption noticing that $|I|<i$ since $|M|>|A|$. Recalling that if $M$ ends with $A$, then $\phi(M)$ starts with $B$, we see that the last line is also equal to
$$2^i-\sum_{k\geqslant 1}\hspace{0.1cm}\sum_{\substack{M=(E_1,\ldots,E_k)\neq (B)\\\hbox{\scriptsize{$E_1\mbox{ starts with }B$}}}}\hspace{0.3cm}\sum_{\substack{I=(0,i_1,\ldots,i_{k})\\ |I|+|M|=i+|B|}}|L^M(I)| \; = \; L^{(B)}(0,i).$$
The equality $L^{(A,A)}(0,i,0)=L^{(B,B)}(0,i,0)$ is proved with a similar argument.
\end{proof}

   {
\small
 
\bibliographystyle{plain}
\bibliography{biblio.bib}

\begin{thebibliography}{10}

\bibitem{Blom84}
Gunnar Blom.
\newblock On the stochastic ordering of waiting times for patterns in sequences of random digits.
\newblock {\em Journal of applied probability}, 21(4):730--737, 1984.

\bibitem{Mansson99}
Isa Cakir, Ourania Chryssaphinou, and Marianne M{\aa}nsson.
\newblock On a conjecture by eriksson concerning overlap in strings.
\newblock {\em Combinatorics, Probability and Computing}, 8(5):429--440, 1999.

\bibitem{Zeilberger24}
Shalosh~B. Ekhad and Doron Zeilberger.
\newblock How to answer questions of the type: If you toss a coin $n$ times, how likely is {HH} to show up more than {HT}?
\newblock {\em arxiv:2405.13561}, 2024.

\bibitem{Eriksson97}
Kimmo Eriksson.
\newblock Autocorrelation and the enumeration of strings avoiding a fixed string.
\newblock {\em Combinatorics, Probability and Computing}, 6(1):45–48, 1997.

\bibitem{Grimmett24}
Geoffrey~R. Grimmett.
\newblock {A}lice and {B}ob on $\bold{X}$: reversal, coupling, renewal, 2024.

\bibitem{Guibas-Odlyzko-81}
Leonidas~J Guibas and Andrew~M Odlyzko.
\newblock String overlaps, pattern matching, and nontransitive games.
\newblock {\em Journal of Combinatorial Theory, Series A}, 30(2):183--208, 1981.

\bibitem{Li-80}
Shuo-Yen~Robert Li.
\newblock A martingale approach to the study of occurrence of sequence patterns in repeated experiments.
\newblock {\em Annals of Probability}, 8(6):1171--1176, 1980.

\bibitem{LittX}
Daniel Litt.
\newblock \url{https://x.com/littmath/status/1769044719034647001}, March 16 2024.

\bibitem{Mansson02}
Marianne Månsson.
\newblock Pattern avoidance and overlap in strings.
\newblock {\em Combinatorics, Probability and Computing}, 11(4):393–402, 2002.

\bibitem{NicaYouTube}
Mihai Nica.
\newblock {T}he {C}oin {F}lip {G}ame that {S}tumped {T}witter: {A}lice {HH} vs {B}ob {HT}.
\newblock \url{https://www.youtube.com/watch?v=BAiuFOwhAWw&t=13s}.

\bibitem{Nica}
Mihai Nica.
\newblock Alice {HH} vs {B}ob {HT} using an {E}dgeworth expansion.
\newblock \url{https://sites.math.rutgers.edu/~zeilberg/mamarim/mamarimhtml/MihaiNicaAliceBob.pdf}, June 11 2024.

\bibitem{Penney-69}
Walter Penney.
\newblock {{P}enney-{A}nte, P}roblem 95.
\newblock {\em Journal of Recreational Mathematics}, 2(4), October 1969.

\bibitem{RameshX1}
Sridhar Ramesh.
\newblock \url{https://x.com/RadishHarmers/status/1770217475960885661}, March 19 2024.

\bibitem{RameshX2}
Sridhar Ramesh.
\newblock \url{https://x.com/RadishHarmers/status/1770530578179198981}, March 20 2024.

\bibitem{Segert24+}
Simon Segert.
\newblock \url{https://github.com/SimonSegert/simonsegert.github.io/blob/main/misc/coin_tossing.pdf}, 2024.

\bibitem{Segert24}
Simon Segert.
\newblock A proof that {HT} is more likely to outnumber {HH} than vice versa in a sequence of $n$ coin flips.
\newblock {\em arxiv:2405.16660}, 2024.

\bibitem{SelbyX}
Alex Selby.
\newblock \url{https://x.com/alexselby1770/status/1769795239667994765}, March 20 2024.

\bibitem{MalloX}
D.L. Yonge-Mallo.
\newblock \url{https://x.com/dlyongemallo/status/1772280170096775170}, March 25 2024.

\bibitem{Zeilbergerjournal}
Doron Zeilberger.
\newblock Personal (online) journal.
\newblock \url{https://sites.math.rutgers.edu/~zeilberg/mamarim/mamarimhtml/litt.html}.

\end{thebibliography}
}

For the ease of reference, we reproduce here\footnote{``Consider a random walk in which one takes equally likely steps of one unit up or one unit down, but with different distributions of speeds. (E.g., maybe up steps take one hour, while down steps have probability 1/2 of taking 2 hours, 1/4 of taking 3 hours, 1/8 of 4 hours, etc).
The time it takes to return to the origin is independent of whether the first step is up and last step is down or vice versa, as doing the same steps in reverse order has the same probability. Thus, for any fixed walk time, the last step away from the origin begun before the time limit is equally likely to be up or down. Thus, at the end when "the buzzer goes off", one is equally likely to be above or below the origin (possibly in the middle of an uncompleted step).
Applied to our game, with HH as a step up in one unit of time and HT$^n$H  as a step down over $n+1$ units of time, this says we are equally likely to end above the origin (Alice wins or we are in the middle of an HT$^n$H step which has tied the game) or below it (Bob wins).
Since it is indeed possible to end in the middle of a game-tying HT$^n$H step (e.g., if the game consists of HHT followed by all T's), Alice is less likely to win than Bob.
QED.
The salient difference is that the "buzzer" can cut off HT$^n$H in the middle (after awarding Bob a game-tying point but before returning to H), which it cannot do for HH. The random walk framing perhaps allows some ready generalization to other interesting problems.''} \emph{verbatim} the tweets \cite{RameshX1} with the permission of their author Sridhar Ramesh.

\end{document}